\documentclass{article}
\usepackage{amssymb}
\usepackage[latin1]{inputenc} 
\usepackage{graphicx}
\usepackage[latin1]{inputenc}
\usepackage{dsfont}
\usepackage{amsmath}
\usepackage{amsthm}
\usepackage{comment}
\newtheorem{theorem}{Theorem}

\newcommand{\wabs}[1]{\left|#1\right|}

\newcommand{\wc}{\mathds{C}}

\newcommand{\wcal}[1]{\mathcal{#1}}

\newcommand{\wfc}[2]{{#1}\!\left(#2\right)}

\newcommand{\wi}[1]{\wrm{i}}

\newcommand{\wlr}[1]{\left( #1 \right)}

\newcommand{\wn}{\mathds N}

\newcommand{\wnorm}[1]{\left\| {#1} \right\|}

\newcommand{\wpade}[1]{Pad\'{e}}

\newcommand{\wref}[1]{(\ref{#1})}
\newcommand{\wrn}[1]{{\mathds R}^{#1}}
\newcommand{\wrm}[1]{\mathrm{#1}}

\newcommand{\wseq}[2]{{\left\{ {#1}_{#2}, \ {#2} \in \wn \right\}}}

\newcommand{\wset}[1]{{\left\{ #1 \right\}}}

\newcommand{\wvec}[1]{\mathbf{#1}}

\begin{document}
\title{Robust Pad\'{e} approximants may have spurious poles}

\author{Walter F. Mascarenhas\thanks{
Instituto de Matem\'{a}tica e Estat\'{i}stica, Universidade de S\~{a}o Paulo,
           Cidade Universit\'{a}ria, Rua do Mat\~{a}o 1010, S\~{a}o Paulo SP, Brazil. CEP 05508-090.
              Tel.: +55-11-3091 5411, Fax: +55-11-3091 6134,  walter.mascarenhas@gmail.com, 
              Supported by grant 2013/10916-2 from Funda\c{c}\~{a}o de Amparo \`{a} Pesquisa do Estado de S\~{a}o Paulo (FAPESP)}
}
\maketitle

\begin{abstract}
We show that robust \wpade{} approximants obtained with the SVD may have spurious poles and
may not converge pointwise.
\end{abstract}


\section{Introduction}

In the recent article \cite{GONNET_REVIEW}, Gonnet et al. explain how to use the SVD to compute
Padé approximants in floating point arithmetic or for problems with noise,
and a similar strategy was proposed in \cite{OLGA}.
According to \cite{GONNET_REVIEW}, ``we can reduce the effects of noise, whether intrinsic
to the data or introduced by rounding errors,'' by decreasing the degree of the approximants
when we detect that some singular values
of a certain matrix are below a threshold. Experiments suggest that these techniques lead
to fewer Froissart doublets and better approximants.

In view of the success of their method in practice, Gonnet et al. asked
whether, from the theoretical point of view, their technique leads
to rational approximants for which we can prove pointwise convergence.
If this were the case then these approximants
would have better theoretical properties than \wpade{}'s, because
we only have proofs of convergence in capacity for \wpade{} approximants
and there are examples in which they diverge for all $z \neq 0$ (see \cite{WALLIN}.)

In this short note we present an adaptation of the classic example by Gammel  \cite{BAKER}
which shows that the answer to the open question of Gonnet et al. is negative.
In fact, in the next section we show that a sequence of approximants
generated by the techniques proposed in \cite{GONNET_REVIEW}
may not converge pointwise to the function being approximated,
and may have spurious poles even when the matrices involved
in their computation have condition number smaller than 5.

\section{Divergence and spurious poles with well conditioned matrices $\wvec{B}_n$}

We consider only \wpade{} approximants of the form $r_{nn} = p_{n}/q_{n}$ and write
\[
\wfc{p_{n}}{z} = \sum_{j = 0}^n a_j z^j
\hspace{1cm} \wrm{and} \hspace{1cm}
\wfc{q_{n}}{z} = \sum_{j = 0}^n b_j z^j,
\]
with $\wvec{a} := \wlr{a_0,\dots,a_n}^t$ and
$\wvec{b} := \wlr{b_0,\dots,b_n}^t$.
To compute the  approximant for  $f = \sum_{k =0}^\infty c_k z^k$ we solve the system
$\wvec{B_n} \wvec{b} = 0$ and set $\wvec{a} = \wvec{A_n} \wvec{b}$, for
\[
\wvec{A_n} :=
\left[
\begin{array}{cccc}
c_{0}  & 0        & \dots & 0 \\
c_{1}  & c_{0}    & \dots & 0 \\
\vdots &  \vdots  & \ddots & \vdots \\
c_{n}  & c_{n- 1} & \dots & c_0 \\
\end{array}
\right],
\]
and
\begin{equation}
\label{def_ab}
\wvec{B}_n :=
\left[
\begin{array}{ccccccc}
c_{n+1}  & c_{n}      & c_{n-1} & \dots   & \dots   & c_1 \\
c_{n+2}  & c_{n+1}    & c_n     & c_{n-1} & \dots   & c_2 \\
\vdots   & \vdots     &         & \ddots  & \ddots  & \vdots \\
c_{2n-1} &            &         &         & c_n     & c_{n-1} \\
c_{2n}   & c_{2 n- 1} &         &         &         & c_{n}
\end{array}
\right].
\end{equation}

Gonnet et al. \cite{GONNET_REVIEW} are concerned
with the case in which $\wvec{B}_n$ is ill-conditioned. In this case, they propose
techniques to reduce $n$ and obtain better approximants. Unfortunately, the ill-conditioning
of $\wvec{B}_n$ is not the only cause of divergence of \wpade{} approximants, and
\cite{GONNET_REVIEW} solves only part of the problem. Of course,
this part is important in practice. However, we now
present examples showing that more is needed in order to eliminate spurious poles
and obtain rational approximants which converge pointwise to $f$.

Our examples are related to  the classic one by Gammel \cite{BAKER}, in which $f$ has the form
\begin{equation}
\label{gammels}
\wfc{f}{z} = 1 + \sum_{k = 1}^\infty \alpha_k \wlr{ \sum_{n = n_k}^{2 n_k} \wlr{z/z_k}^n}
= 1 + \sum_{k = 0}^\infty \alpha_k \frac{\wlr{z/z_k}^{n_k} - \wlr{z/z_k}^{2 n_k + 1}}{1 - z/z_k},
\end{equation}
where $n_k = 2^k - 1$ and the $\alpha_k$ are chosen in order to enforce the convergence of the
series above. Gammel's example does not yield a negative answer to the open question by Gonnet et al.,
because the matrices $\wvec{B}_n$ they lead to may be ill-conditioned. However, if instead
of asking for an entire function we content ourselves with  $f$ for which the series $\wfc{f}{z} = \sum_{j = 0}^\infty c_j z^j$
converges for $\wabs{z} < 1$, then the function $f$ with the slightly different form
\begin{equation}
\label{mine}
\wfc{f}{z} := 1 + \sum_{k = 2}^\infty 16^k \wlr{z^{n_k - 1} + z_k^{2 n_k} \frac{\wlr{z/z_k}^{n_k} - \wlr{z/z_k}^{2 n_k + 1}}{1 - z/z_k}},
 \ n_k := 2^k - 2,  \ \
\end{equation}
for any sequence $\wseq{z}{k}$ with $0 < \wabs{z_k} < 1/3$,
leads to an example in which the matrix $\wvec{B}_n$ is well-conditioned, the denominator $\wfc{q_{n_k}}{z}$
is equal to $1 - z/ z_k$ and $\wfc{p_{n_k}}{z_k} \neq 0$.
The $(n_k,n_k)$ \wpade{} approximant of the function $f$ in \wref{mine} has a spurious pole at $z_k$,
and the sequence of approximants do not converge uniformly in any set $A$ such that the interior of
$A_{1/3} := A \cap \wset{\wabs{z} < 1/3}$ is non empty and
$\wseq{z}{k}$ is dense in $A_{1/3}$. Moreover,
when each element in the sequence $\wseq{z}{k}$ is repeated infinitely many times,  we
do not have pointwise convergence at the $z_k$, because the $(n_k, n_k)$ \wpade{} approximant assumes
the value $\infty$ at $z_k$. For example, if the
points $z_k$ are
\[
\frac{1}{4}, \ \frac{1}{4}, \frac{1}{5}, \ \frac{1}{4}, \frac{1}{5},\frac{1}{6}, \
\dots\ , \
\ \frac{1}{4}, \frac{1}{5},\frac{1}{6}, \frac{1}{7}, \ \dots \ , \ \frac{1}{n},
\frac{1}{4}, \frac{1}{5},\frac{1}{6}, \frac{1}{7}, \ \dots\ , \ \frac{1}{n}, \frac{1}{n + 1}, \ \dots
\]
then we do not have pointwise convergence at anyone of them.

We now formalise the arguments above, and after that we present
our acknowledgements and comments regarding a related article.

\begin{theorem} \label{thm_main}
If the points $z_k \in \wc{}$ are such that $0 < \wabs{z_k} < 1/3$, then the
coefficients $c_j$ in the expansion $\wfc{f}{z} = \sum_{j=0}^\infty c_j z^j$ of the
function $f$ in \wref{mine} are such that
$0 < \wabs{c_j} \leq \wlr{j + 3}^4$  for $j = 2, 3, 4, \dots$ and
\begin{itemize}
\item[(i)] The function $f$ has a $(n_k,n_k)$ \wpade{} approximant
with $\wfc{q_{n_k}}{z} = 1  - z / z_k$ and $\wfc{p_{n_k}}{z_k} \neq 0$.
\item[(ii)] The singular values $\wfc{\sigma_1}{\wvec{B}_n}$ and $\wfc{\sigma_n}{\wvec{B}_n}$ of the
matrix $\wvec{B}_n$ in \wref{def_ab} corresponding to $n = n_k$ satisfy
$ \wfc{\sigma_1}{\wvec{B}_n} <  5 \wfc{\sigma_n}{\wvec{B}_n}$.\\
\end{itemize}
\end{theorem}

\begin{proof}
Inspecting \wref{mine}, the reader will notice that
$c_0 = 1$ and, for $k = 2,3 \dots$,
\begin{equation}
\label{def_ck}
c_{2^k - 3} = 16^k \hspace{0.5cm} \wrm{and} \hspace{0.5cm}
c_j = 16^k z_k^{2^{k+1} - 4 - j} \ \ \wrm{for} \ \ 2^{k} - 2 \leq j \leq 2^{k + 1} - 4.
\end{equation}
Since $\wabs{z_k} < 1$ it follows that,
for $j$ in \wref{def_ck}, we have
$\wabs{c_j} \leq 16^k = (2^k)^4 \leq \wlr{j + 3}^4$.

We can rewrite the function $f$ in \wref{mine} as
\[
\wfc{f}{z} = \frac{\wfc{p_{n_k}}{z}}{1 - z/z_k} + \wfc{O}{z^{2 n_k + 1}}
\]
for
\[
\wfc{p_{n_k}}{z} := \wlr{1 + \sum_{j = 2}^{n_k - 1} 16^j \wlr{z^{n_j - 1} +
z_k^{2 n_k} \sum_{\ell = n_j}^{2n_j} \wlr{\frac{z}{z_j}}^\ell}} \wlr{1 - \frac{z}{z_k}}
\]
\[
+ 16^k \wlr{z^{n_k - 1} - \frac{z^{n_k}}{z_k} + \wlr{z z_k}^{n_k}},
\]
and the accuracy-through-order criterion shows that $\wfc{q_{n_k}}{z} = 1 - z / z_k$ and $\wfc{p_{n_k}}{z}$
define the $(n_k, n_k)$ Padé approximant of $f$. It follows that $\wfc{p_{n_k}}{z_k} = 16^{k} z_k^{2 n_k} \neq 0$,
and we have proved item (i) in Theorem \ref{thm_main}.

In order to verify item (ii), let us show that the matrix $\wvec{B}_n$ corresponding to $n = n_k = 2^k - 2$ has appropriate singular values.
Equation \wref{def_ck} shows that $c_{n-1} = c_{2^k -3} = c_{2n} = 16^k$, and we can write
\[
\wvec{B}_n = 16^k \wvec{U} + \wvec{V} \hspace{1cm} \wrm{with} \hspace{1cm} \wvec{V} := c_n \wvec{V}_0 + \sum_{j = 2}^{n-1} c_{n - j} \wvec{V}_j + \sum_{j = 1}^{n-1} c_{n +j} \wvec{W}_j,
\]
where $\wvec{U}$, $\wvec{V}_j$ and $\wvec{W}_j$  are $n \times (n + 1)$ matrices such that
\begin{itemize}
\item $u_{n1} = u_{i (i + 2)} = 1$ for $i = 1,\dots n-1$, and $u_{ij} = 0$ otherwise,
\item $(v_j)_{i (i + j + 1)} = 1$ for $ i = 1, \dots n - j$, and $(v_j)_{ik} = 0$ otherwise,
\item $(w_j)_{i (i - j  + 1)} = 1$ for $ i = j, \dots n$, and $(w_j)_{ik} = 0$ otherwise.
\end{itemize}
The matrix $\wvec{U}$ has singular values $\wfc{\sigma_1}{\wvec{U}} = \wfc{\sigma_2}{\wvec{U}} = \dots = \wfc{\sigma_n}{\wvec{U}} = 1$, because
$\wvec{U} \wvec{U}^t = \wvec{I}_n$,  and $\wnorm{\wvec{V}_j}_2 = \wnorm{\wvec{W}_j}_2 = 1$. It follows that
\begin{equation}
\label{def_s}
\wnorm{\wvec{V}}_2 \leq S := \wabs{c_n} + \sum_{j = 1}^{n - 2}\wabs{c_j} +  \sum_{j = n + 1}^{2n - 1} \wabs{c_j} =
\sum_{j = 1}^{n - 2}\wabs{c_j} +  \sum_{j = n}^{2n - 1} \wabs{c_j}
\end{equation}
and
\begin{equation}
\label{eq_sigma}
16^k - S \leq \wfc{\sigma_n}{\wvec{B}_n} \leq \wfc{\sigma_{1}}{\wvec{B}_n} \leq 16^k + S,
\end{equation}
due to the inequality $\wabs{\wfc{\sigma_i}{\wvec{M} + \bf{\Delta}} - \wfc{\sigma_i}{\wvec{M}}} \leq \wnorm{\bf{\Delta}}_2$,
which follows from the min/max characterization of the singular values of the $n \times m$ matrix $\wvec{M}$ with $m \geq n$:
\[
\wfc{\sigma_i}{\wvec{M}} = \min_{\wfc{\wrm{dim}}{\wcal{S}} = m - i + 1}
\wlr{ \max_{\wvec{x} \in \wcal{S}\  \wrm{with} \ \wnorm{\wvec{x}}_2 = 1} \wnorm{\wvec{M} \wvec{x}}_2 },
\]
where the minimum is taken over all subspaces $\wcal{S}$ of $\wrn{m}$ with dimension $m - i + 1$.
Recalling that $n = 2^k - 2$ and $\wabs{z_k} < 1/3$, we deduce from \wref{def_ck} that
\begin{equation}
\label{first_sum}
\sum_{j = n}^{2n - 1} \wabs{c_j} = \sum_{j=n}^{2n - 1} 16^k \wabs{z_k}^{2^{k+1} - 4 - j}
= 16^k \sum_{j=2^k - 2}^{2^{k+1} - 5}  \wabs{z_k}^{2^{k+1} - 4 - j}
\end{equation}
\[
= 16^k \sum_{m=1}^{2^k - 2} \wabs{z_k}^m < 16^k \sum_{m=1}^{\infty} \wabs{z_k}^m =
\frac{16^k \wabs{z_k}}{1 - \wabs{z_k}}
< \frac{16^k}{2}.
\]
Moreover, using \wref{def_ck} we obtain
\begin{equation}
\label{second_sum}
\sum_{j = 1}^{n - 2}\wabs{c_j} =
\sum_{m = 2}^{k - 1} \sum_{j = 2^m - 3}^{2^{m+1} - 4}\wabs{c_j} =
\sum_{m = 2}^{k - 1} 16^m \wlr{1 + \sum_{j = 2^m - 2}^{2^{m+1} - 4} z_m^{2^{m+1} - 4 - j} }
\end{equation}
\[
= \sum_{m = 2}^{k - 1} 16^m \wlr{ 1 + \sum_{l = 0}^{2^m - 2} \wabs{z_m}^{l} } <
\sum_{m = 2}^{k - 1} 16^m \wlr{ 1 + \sum_{l = 0}^{\infty} \wabs{z_m}^{l} }
\]
\[
= \sum_{m = 2}^{k - 1} 16^m \wlr{1 + \frac{1}{1 - \wabs{z_m}}} < \frac{5}{2} \sum_{m = 2}^{k-1} 16^m
< \frac{5}{2} \frac{16^k}{15} = \frac{16^k}{6}.
\]
Equations \wref{def_s}, \wref{first_sum} and \wref{second_sum} show that $S < 2 \times 16^k / 3$. Thus,
\wref{eq_sigma} yields
\[
\frac{\wfc{\sigma_1}{\wvec{B}_n}}{\wfc{\sigma_n}{\wvec{B}_n}} < \frac{1 + \frac{2}{3}}{1 - \frac{2}{3}} = 5
\]
and we are done.
$\qquad$ \end{proof}

\section{Acknowledgments and related work}
We would like to thank both referees for reviewing our work. In particular,
the first referee called our attention for the similarity between our
example and Gammel's. The first version of this note was based only in \cite{WALLIN},
and our function $f$ was described in terms of the coefficients $c_j$.
We were not aware of the explicit formula \wref{gammels}, which lead us
to write our examples as in equation \wref{mine} in the present version. This formula gives a
clearer view of the block structure of the Padé approximants in our example, and
this contribution by the first referee was much appreciated.

Finally, we would like to call the reader's attention to the
related article \cite{MATOS}, which presents the results of numerical experiments
regarding the specific problem we discuss here, and proves theoretical results
regarding several problems that we did not address. \cite{MATOS} does not
present theoretical examples like ours, but it is broader than this short
note.


\begin{thebibliography}{10}

\bibitem{BAKER} G. A Baker Jr. and P. Graves-Morris, {\em Padé Approximants, Part I: Basic Theory},
Encyclopedia of Mathematics and its applications, Addison Wesley, 1981.

\bibitem{MATOS} B. Beckermann and A. C. Matos,  {\em Algebraic properties of robust Padé approximants},
published online in June 2014 by J. Approx. Theory, DOI: 10.1016/j.jat.2014.05.018.

\bibitem{GONNET_REVIEW}
P. Gonnet, S. G\"{u}ttel and L.~N. Trefethen, {\em Robust \wpade{} Approximation via SVD},
SIAM Review, vol. 35, n 1, (2013), pp.~101--117.

\bibitem{OLGA}
O.~L. Ibryaeva and V.~M. Adukov, {\em An algorithm for computing a \wpade{} approximant with minimal degree denominator},
J. Comput. Appl. Math., 237, (2013), pp.~529--541.

\bibitem{WALLIN} H. Wallin, {\em The Convergence of \wpade{} Approximants and the Size of the Power Series Coefficients},
Appl. Anal., 4, (1974), pp.~235--251.
\end{thebibliography}
\end{document}